\documentclass[12pt]{amsart}
\usepackage{color}
\usepackage{amsmath}

\def\bee{\begin{equation*}}
\def\eee{\end{equation*}}
\def\aint{\frac{\ \ }{\ \ }{\hskip -0.4cm}\int}

\def\e{\epsilon}
\def\lf{\left}

\def\ijb{i\bar{j}}
\def\ri{\right}
\def\hg{\widehat{g}}
\def\a{{\alpha}}

\def\wt{\widetilde}

\def\p{\partial}
\newcommand\R{{\mathbb R}}

\newcommand\C{{\mathbb C}}

\def\jbar{{\bar\jmath}}

\def\K{K\"ahler }
\def\KR{K\"ahler-Ricci }
\def\KRF{K\"ahler-Ricci flow }

\def\A{Amp\`{e}re }

\def\be{\begin{equation}}
\def\ee{\end{equation}}

\def\lf{\left}
\def\ri{\right}
\def\a{{\alpha}}

\def\txi{\widetilde \xi}

\def\e{\epsilon}

\def\ijb{{i\jbar}}

\def\Ric{\text{\rm Ric}}

\def\wt{\widetilde}

\def\p{\partial}

\def\C{\Bbb C}

\def\wt{\widetilde}

\def\p{\partial}

\def\p{\partial}

\def\C{\Bbb C}

\def\KRF{K\"ahler-Ricci flow }

\newtheorem{thm}{Theorem}[section]

\newtheorem{lem}{Lemma}[section]
\newtheorem{prop}{Proposition}[section]
\newtheorem{cor}{Corollary}[section]
\theoremstyle{definition}
\newtheorem{defn}{Definition}[section]
\theoremstyle{remark}
\newtheorem{rem}{Remark}[section]
\numberwithin{equation}{section}

\begin{document}

 \title{\bf Longtime existence of the K\"ahler-Ricci flow on $\C^n$}\vskip .2cm
\author{Albert Chau$^1$}
\address{Department of Mathematics,
The University of British Columbia, Room 121, 1984 Mathematics
Road, Vancouver, B.C., Canada V6T 1Z2} \email{chau@math.ubc.ca}

\author{Ka-Fai Li}

\address{Department of Mathematics,
The University of British Columbia, Room 121, 1984 Mathematics
Road, Vancouver, B.C., Canada V6T 1Z2} \email{kfli@math.ubc.ca}

\author{Luen-Fai Tam$^2$}

\thanks{$^1$Research
partially supported by NSERC grant no. \#327637-06}
\thanks{$^2$Research partially supported by Hong Kong RGC General Research Fund  \#CUHK 14305114}

\address{The Institute of Mathematical Sciences and Department of
 Mathematics, The Chinese University of Hong Kong,
Shatin, Hong Kong, China.} \email{lftam@math.cuhk.edu.hk}
\thanks{\begin{it}2000 Mathematics Subject Classification\end{it}.  Primary 53C55, 58J35.}
%\thanks{\begin{it}Key words and phrases\end{it}.  Non-compact K\"ahler-Einstein metrics, %K\"ahler-Ricci flow, parabolic  Monge-Amp\`{e}re  equation.}

\begin{abstract} We produce longtime solutions to
the K\"ahler-Ricci flow for complete \K metrics on $\C^n$ without assuming the initial metric has bounded curvature, thus extending results in \cite{ChauLiTam}.  We prove the existence of a longtime bounded curvature solution emerging from any complete $U(n)$-invariant \K metric with non-negative holomorphic bisectional curvature, and that the solution converges as $t\to \infty$ to the standard Euclidean metric after rescaling.  We also prove longtime existence results for more general \K metrics on $\C^n$ which are not necessarily $U(n)$-invariant.

\end{abstract}

\maketitle\markboth{Albert Chau, Ka-Fai Li and Luen-Fai Tam} {Longtime existence of the K\"ahler-Ricci flow on $\C^n$}

%\tableofcontents
\section{Introduction}

   The \K Ricci flow is the following evolution equation for an initial \K metric $g_0$ on a complex manifold $M$:

 \be\label{krf}
 \left\{
   \begin{array}{ll}
     \displaystyle\frac{\partial g_{\ijb}}{\partial t} =-R_{\ijb}\\
          g(0)  = g_0.
   \end{array}
 \right.
 \ee
In this paper we establish longtime existence results for \eqref{krf} in the case $M=\C^n$.  For simplicity we will always assume $g_0$ to be smooth, although our results will only require $g_0$ to be \K and in some cases, also twice differentiable while satisfying some conditions.  We are particularly interested in complete \K metrics with non-negative bisectional curvature.

One of our main results is  (see Definition \ref{definition1} below):

\begin{thm}\label{T1}
 Let $g_0$ be a complete $U(n)$-invariant \K metric on $\C^n$ with non-negative bisectional curvature.  Then
 \begin{enumerate}
   \item [(i)] the \KRF \eqref{krf} has a smooth     longtime $U(n)$-invariant solution $g(t)$ which is equivalent to $g_0$ and has bounded non-negative bisectional curvature;
   \item [(ii)] $g(t)$ converges, after rescaling at the origin, to the standard Euclidean metric on $\C^n$;
   \item [(iii)] $g(t)$ is unique in the class $\mathcal{S}(g_0)$.
 \end{enumerate}

\end{thm}
See Corollary \ref{c-nonnegative-longtime},  Theorem \ref{t-Un-limit-1}, Theorem \ref{t-uniqueness-2} for details.  Here and throughout the rest of the paper, we use the following notation in reference to a solution $g(t)$ to \eqref{krf} on $M\times [0, T)$:

\begin{defn}\label{definition1}
$\bullet$ We say the solution is \begin{sl}smooth\end{sl} if $g(t)\in C^{\infty}(M\times (0, T) \bigcap C^0(M\times [0, T))$.

$\bullet$ We call the solution a \begin{sl}longtime solution\end{sl} if $T=\infty$.

$\bullet$ We say the solution is \begin{sl}complete \end{sl} if $g(t)$ is complete for each $t\in (0, T)$.

 $\bullet$ We say  that the solution is \begin{sl}equivalent to $g_0$\end{sl} if for all $0\leq T_1\leq t \leq T_2<T$ we have $c^{-1}g_0\leq g(t)\leq c g_0$ for some $c>0$.

$\bullet$ We say the solution has \begin{sl}non-negative bisectional curvature\end{sl} if $g(t)$ has non-negative bisectional curvature for all $t\in[0, T)$, and the solution has \begin{sl}bounded curvature\end{sl} if the curvature of $g(t)$ is uniformly bounded on $M\times[T_1, T_2]$ for all $0<T_1\leq T_2<T$.

$\bullet$ We say that a longtime solution \begin{sl}converges to $\widetilde{g}$ after rescaling at $p\in M$\end{sl} if for some $V\in T_p M$ the metrics $\frac{1}{|V|_t^2}g(t)$ converge to $\widetilde{g}$ smoothly and uniformly on compact subsets of $M$ where $|V|_t^2=g(t)(V, V)$.

$\bullet$     $\mathcal{S}(g_0)$ denotes the set of solutions of the \KR flow with initial data which is uniformly equivalent to $g_0$.
\end{defn}

 We want to remark that using the result of Cabezas-Rivas and Wilking \cite{CW}, Yang and Zheng \cite{YZ} has produced a short time existence of the \KR flow for $U(n)$ invariant initial data $g_0$ with nonnegative sectional curvature. See also \cite{HuangTam}. However, it is unclear if the solution is uniformly equivalent to $g_0$. Hence it is unclear if the solution is the same as the one constructed in the above Theorem.

 The short time existence of $g(t)$ in the above Theorem was proved by Chau-Li-Tam in \cite{ChauLiTam}, while the fact that the solution has non-negative bisectional curvature was proved by Yang-Zheng in \cite{YZ}.  The existence of complete $U(n)$-invariant \K metrics with unbounded non-negative bisectional curvature was shown in \cite{WZ} and thus, in the $U(n)$-invariant case, our results extend the longtime existence results of Shi \cite{Shi2}.  These results are all motivated by Yau's uniformization conjecture stating that: if $g_0$ is a complete non-compact \K manifold with positive holomorphic bisectional curvature then $M=\C^n$ and we refer to \cite{ChauTam2007}, and the references therein, for a survey of Yau's conjecture and the \KRF.

 In the case of non-negative bisectional curvature, the volume growth of $g_0$ is of particular importance in the proofs above.  Let $V(r)$ be the volume of the geodesic $r$ ball around the origin relative to $g_0$ having non-negative bisectional curvature.  It was shown by  Chen-Zhu \cite{CZ} that $$\frac{1}{c}r^n\leq V(r)\leq cr^{2n}$$ for some constant $c$.  Following the notation in \cite{WZ}, we say $(\C^n, g_0)$ is a conoid if    $\limsup_{r\to \infty}V(r)/r^{2n}$  is positive, and   that it  is a cigar if   $\limsup_{r\to \infty}V(r)/r^{n}<\infty$.  The case when $g_0$ is a cigar is rather special, and our arguments rely heavily on the non-negativity of bisectional curvature.  When $g_0$ is not a cigar, we may generalize the above theorem to more general \K metrics which may not have nonnegative bisectional curvature, and may not even be $U(n)$-invariant.

\begin{thm}\label{T2}
Let $g_0$ be \K metric equivalent to a $U(n)$-invariant \K metric which has non-negative bisectional curvature and is not a cigar.  Suppose $g_0$ is either $U(n)$-invariant, or else has bounded curvature.  Then the \KRF \eqref{krf} has a  unique smooth longtime solution $g(t)$ which is equivalent to $g_0$, has bounded curvature and is $U(n)$-invariant provided $g_0$ is $U(n)$-invariant.  The solution is  unique in the class $\mathcal{S}(g_0).$
\end{thm}
We refer to Theorems \ref{t-Un-longtime-2} and \ref{t-Un-longtime-2-2} for details where this and more general results are proved.

 We point out that when $n=1$, there always exists a longtime solution to \eqref{krf} by the results of Giesen-Topping \cite{GT}.  In fact, it is shown there that given any non-compact Riemann surface $(M, g_0)$, the \K Ricci flow \eqref{krf} admits a smooth complete longtime solution $g(t)$ with unbounded curvature for all $t>0$.  Shorttime and longtime existence results for the Ricci flow starting from complete non-compact Riemannian manifolds $(M, g_0)$ with unbounded curvature have appeared in the works \cite{CW, GT, KL, SS1, SS2}. We also point out that in \cite{CW} examples of immortal solutions $g(t)$ (defined for all $t\in \R$) were constructed in \cite{CW}.  We refer to \cite{To} for a survey of related results.

 The outline of the paper is as follows.  In \S 2 we establish a uniqueness result for \eqref{krf} on general \K manifolds. In section \S 3 we prove our main longtime existence results on $\C^n$ (Theorems \ref{t-Un-longtime-2} and \ref{t-Un-longtime-2-2}) including our longtime existence result for $U(n)$-invariant \K metrics with non-negative bisectional curvature (Corollary \ref{c-nonnegative-longtime}).  In \S 4 we prove our convergence result for longtime $U(n)$-invariant solutions with non-negative bisectional curvature (Proposition \ref{p-limit-1}). In the appendix, we collect some basic facts of $U(n)$-invariant \K metrics on $\C^n$ for easy reference.

\section{Uniqueness}

Our first result is on the uniqueness of solutions to the \KR flow \eqref{krf}. It is well-known that the Ricci flow on complete noncompact Riemannian manifold is unique under the assumption that the curvature is bounded, see \cite{CZ2}.  For the \KR flow, it is more easy to obtain uniqueness. The following is a generalization of the result in \cite{F}. Here we do not assumption the curvature is bounded, and we do not assume that the Ricci form has a potential, which is assumed in \cite{F}.

\begin{thm}\label{t-uniqueness-2}
Let $(M^n,\widehat g)$ be a complete noncompact \K manifold.
Suppose
 there is an exhaustion function $\zeta>0$  on $(M^n,\widehat g)$ with $\lim_{x\to\infty}\zeta(x)=\infty$ such that $|\p\bar\p\zeta|_{\hg}$ and $|\widehat\nabla \zeta|_{\hg}$ are bounded.

 Let $g_1(x,t)$ and $g_2(x,t)$ be two solutions of the \KR flow \eqref{krf} on $M\times[0,T]$ with the same initial data $g_0(x)=g_1(x,0)=g_2(x,0)$. Suppose there is a positive function $\sigma$ with $\lim_{x\to\infty}\log\sigma(x)/\log\zeta(x)=0$ such that the following conditions hold for all $(x,t)\in M\times[0,T]$:

  \begin{enumerate}
     \item [(i)]
 \be\label{eq-bound-2}
  \hg(x) \le \zeta(x) g_1(x,t); \ \   \hg(x) \le  \zeta(x) g_2(x,t),
\ee

     \item [(ii)]
     \bee
     -\sigma(x)\le    \frac{\det((g_1)_\ijb (x,t))}{\det(  (g_2)_\ijb (x,t))}\le \sigma(x).
     \eee

   \end{enumerate}
   Then $g_1\equiv g_2$ on $M\times[0,T]$. In particular, if $g_1$ and $g_2$ are uniformly equivalent to $\hg$ on $M\times[0,T]$, then $g_1\equiv g_2$.
 \end{thm}
 \begin{proof} By adding a positive constant to $\zeta$ we may assume that  $\eta:=\log \zeta>1$. Then
\bee
\eta_\ijb=\frac{\zeta_{\ijb}}{\zeta}-\frac{\zeta_i\zeta_{\bar j}}{\zeta^2}.
\eee
Since $|\p\bar\p\eta|_{\hg}$ and $|\widehat\nabla \eta|_{\widehat g}$ are uniformly bounded, there is $c_1>0$ such that
\bee
|\p\bar\p\eta|_{\hg}\le \frac{c_1}{\zeta}
\eee
on $M$.
Let $h(s,t)(x)=sg_1(x,t)+(1-s)g_2(x,t)$, $0\le s\le 1$. By (i), $\hg(x)\le\zeta(x)h(s,t)(x,t)$ for all $(x,t)\in M\times[0,T]$ and for all $s$.  Let $(x,t)$ be fixed and diagonalize $\p\bar\p \eta$ with respect to $\hg$ at $x$. Then $|\eta_{i\bar i}|\le \frac{c_1}{\zeta}$. On the other hand,
\be\label{e-unique-1-1}
\Delta_{h(s,t)}\eta=(h(s,t))^{\ijb}\eta_{\ijb}=(h(s,t))^{i\bar i} \eta_{i\bar i}\le n\zeta\cdot\frac{c_1}{\zeta}=nc_1.
\ee
  Let
 $$
 w(x,t)=\int_0^t\lf(\log \frac{\det((g_1)_\ijb (x,s))}{\det(  (g_2)_\ijb (x,s))} \ri)ds.
 $$
Then
 \bee
 \begin{split}
 w_\ijb(x,t)=&\int_0^t\lf(  ( {R}_1)_\ijb (x,s)-( {R}_2)_\ijb (x,s)\ri)ds=-(g_1)_{\ijb}(x,t)+(g_2)_\ijb(x,t)
 \end{split}
 \eee
 where $ (R_k)_\ijb$ is the Ricci tensor  of $ g_k$, $k=1,2$. Here we have used the \KR flow and the fact that $g_1=g_2$ at $t=0$. Hence in order to prove the proposition, it is sufficient to prove that $w\equiv0$.
 Now
 \be
 \begin{split}
 \frac{\p }{\p t}w(x,t)=&\int_0^1\frac{\p}{\p s}\log\det (h_\ijb(s,t)(x)))ds\\
 =&\int_0^1\Delta_{h(t,s)}w(x,t)ds.
\end{split}
 \ee
Let $W(x,t)=e^{At}\eta$ where $A=nc_1+1 $. By \eqref{e-unique-1-1},

\bee
 \begin{split}
 \frac{\p}{\p t}W(x,t)-\int_0^1\Delta_{h(t,s)}W(x,t)ds\ge&e^{At}(A\eta-nc_1) \\
 \ge &e^{At}\eta
 \end{split}
 \eee
where we have used the fact that $\eta>1$. For any $\e>0$,
\bee
 \begin{split}
 \frac{\p}{\p t}\lf(\e W-w\ri)(x,t) -\int_0^1\Delta_{h(t,s)}\lf(\e W-w\ri)(x,t)ds
 \ge &e^{At}\eta\\
 \end{split}
 \eee
 By (ii), $\lim_{x\to\infty}(\e W-w)(x,t)= \infty$ uniformly in $t$. By the maximum principle, we conclude that $w\le \e W$. Letting $\e\to0$ gives $w\le 0$. Similarly, one can prove that $-w\le 0$ and hence $w\equiv0$ on $M\times[0,T]$. This completes the proof of the proposition.
 \end{proof}

\begin{rem}  \begin{enumerate}
                 \item [(i)] Suppose $\widehat g$ has bounded curvature, then $\zeta$ exists by \cite{Shi2}, see also \cite{T}.
                 \item [(ii)]  Suppose $\widehat g$ has nonnegative Ricci curvature and nonnegative quadratic bisectional curvature, then $\zeta$ exists by \cite{NiTam2013}.
                 \item [(iii)] In particular, the solution constructed in Theorem 4.2 in \cite{ChauLiTam} does not depend on the subsequence chosen.
               \end{enumerate}
 \end{rem}

 \section{Long time existence}

 In this section, we discuss the longtime existence of solutions to \KR flow \eqref{krf} starting from $U(n)$-invariant \K metrics on $\C^n$.  We will make extensive use of the notation and results in Theorem \ref{appendixthm} (see appendix) on $U(n)$-invariant \K metrics $g$ and their associated auxiliary functions $\xi, h, f$, and we refer there for details.  We also recall the notation in definition \ref{definition1}.

 Suppose $g$ is a complete $U(n)$-invariant \K metric on $\C^n$ with nonnegative bisectional curvature.  Then the scalar curvature $\mathcal{R}$ satisfies \cite{WZ}:
 $$
 \frac1{V_g(x,\rho)}\int_{B_g(x,\rho)}\mathcal{R}\le \frac{c(x)}{1+\rho}
 $$
 for all $\rho>0$ for some constant $c(x)$ which may depend on $x$. Note that if we can choose $c(x)=c$ which is independent of $x$ and the curvature of $g$ is bounded, then it is well known \cite{Shi2} (see also \cite{NiTam2004}) that the \KR flow has a longtime solution starting from $g$.  In our   main  result below,  Theorem \ref{t-Un-longtime-2}, we do not assume that the curvature of $g$ is bounded, and $c(x)$ may in general depend on $x$.    In fact, we will prove longtime existence of the flow under more general conditions than nonnegative bisectional curvature (see {\bf(c1)}, {\bf(c2)}, {\bf(c3)} below). In order to obtain long time existence, we make use the following  basic lemma from \cite[Corollary 4.2]{ChauLiTam}, see also the proof of \cite[Theorem 4.2]{ChauLiTam}.

 \begin{lem}\label{l-existencetimeestimate-1}
 Let $M^n$ be a complex noncompact manifold and let $g_0$, $\hg$ be complete \K metrics with  bounded curvature on $M$. Suppose the holomorphic bisectional curvature of $\hg$ is bounded above by $K$ and that  $(1/\e)\widehat{g}\le g_0\le C\widehat{g}$ for some $C\ge1$.  Let $T=1/2nK\e$ if $K>0$, otherwise let $T=\infty$.  Then the \KR flow \eqref{krf} with $g_0$ as initial data has a smooth solution $g(t)$ on $M\times[0,T)$ which has bounded curvature and satisfies
 \be
     \lf(\frac1n- 2 \e Kt\ri)\hg\le g(t)\le B(t)\hg
 \ee
 for all $t\in [0, T)$ for some positive continuous function $B(t)$ on $[0,T)$.
 \end{lem}

 For a given \K metric $g$ and constant $\epsilon>0$, it is very difficult in general to construct a metric $\widehat{g}$ satisfying the conditions in Lemma \ref{l-existencetimeestimate-1} such that $K$ is independent of $\epsilon$.  When $g$ is a $U(n)$-invariant metric generated by $\xi$ (see Appendix), we show this is possible when {\bf(c1)} or {\bf (c2)} below holds (see Proposition \ref{s2p1}), though may not be possible  when {\bf (c3)} below holds (see Proposition \ref{p-nonexitence}).

\begin{description}
  \item[(c1)]  There exist $0\le \a\le\beta<1$ and $\gamma$ such that for all $0<a<r$
  \bee
  \int_a^r\frac{\a-\xi}sds, \ \ \int_a^r\frac{\xi-\beta}sds\le \gamma.
  \eee
  \item[(c2)]   $\lim_{r\to\infty}\xi(r)=1$,  $
 \int_1^\infty\frac{1-\xi(s)}sds=\infty
  $ and there exists  $\delta>0$ such that for all $0<a<r$
  \bee
  \int_a^r\frac{1-\xi}sds \ge -\delta.
  \eee

  \item[(c3)] $\lim_{r\to\infty}\xi(r)=1$, and
  $
  \int_1^\infty\frac{1-\xi(s)}sds<\infty.
  $
\end{description}
In \cite[Theorem 5.4]{ChauLiTam} it was proved that if $\xi$ satisfies {\bf (c1)} with $\beta\le 1$, then the \KR flow has short time existence with initial metric $g$.
 It was also proved in \cite{ChauLiTam} that if the bisectional curvature is nonpositive, then we have long time solutions. In case the bisectional curvature of $g$ is nonnegative, then it satisfies either {\bf (c1)} with  $\a=0$ and $ \beta=\xi(\infty)$, {\bf (c2)}  or {\bf (c3)} by Theorem \ref{appendixthm} (c), where {\bf (c3)} is the case of a cigar and {\bf (c1)} is the case of a conoid. So the conditions are weaker than the assumption that the bisectional curvature is nonnegative.

The main result of this section is:

\begin{thm}\label{t-Un-longtime-2} Let $g_0$ be a complete \K metric on $\C^n$ satisfying one of the following:
\begin{enumerate}
\item [(i)] $g_0$ is $U(n)$-invariant satisfying {\bf (c1)} or {\bf (c2)},
\item [(ii)] $g_0$ has bounded curvature and is equivalent to a $U(n)$-invariant metric satisfying {\bf (c1)} or {\bf (c2)},  or
\item [(iii)] $g_0$ is $U(n)$-invariant satisfying {\bf (c3)}, and has non-negative bisectional curvature.
\end{enumerate}
Then the \KR flow \eqref{krf} with $g_0$ as initial data has a  smooth longtime solution $g(t)$ which is equivalent to $g_0$, has bounded curvature and is $U(n)$-invariant provided $g_0$ is  $U(n)$-invariant.  Moreover  $g(t)$ is unique in the class $\mathcal{S}(g_0)$.

\end{thm}

  In particular, we get the following

\begin{cor}\label{c-nonnegative-longtime} Let $g_0$ be a complete $U(n)$-invariant \K metric on $\C^n$ with nonnegative bisectional curvature.  Then the \KR flow \eqref{krf} with $g_0$ as initial data has a smooth longtime $U(n)$-invariant solution $g(t)$ which is equivalent to $g_0$ and has bounded non-negative bisectional curvature.  Moreover $g(t)$ is unique in the class $\mathcal{S}(g_0)$.

\end{cor}

 In order to prove parts (i) and (ii)  of the theorem, we need the following:

\begin{prop}\label{s2p1}
Let $g$ be a smooth $U(n)$-invariant metric with bounded curvature generated by $\xi$ satisfying {\bf(c1)} or {\bf(c2)}.
Then given any $\epsilon>0$ there exists $\widetilde{g}$ satisfying
\begin{enumerate}
\item [(a)] The curvature of $\widetilde{g}$ is bounded  by a constant independent of $\epsilon$.
\item [(b)] $(1/\epsilon) \widetilde{g} \leq g \leq C \widetilde{g}$ for some constant $C$.
\end{enumerate}
\end{prop}
\begin{rem}
 Suppose $g$ has non-negative bisectional curvature, decaying on average like $dist(p, \cdot)^{-(1+a)}$, uniformly around all points on $\C^n$ for some $a>0$.  Then the Proposition also follows from estimates for the \KRF in \cite{Shi2} and \cite{NiTam2004}.
 \end{rem}
\begin{proof}[Proof of Proposition \ref{s2p1}]
 Suppose $\xi$ satisfies {\bf(c1)}.    For each $k\geq 1$, consider the linear automorphism of $\C^n$ given by $\phi_k (z)=z/\sqrt{k}$ and consider the $U(n)$-invariant \K metric $g_k:=\phi^*_k g$ on $\C^n$.    Consider the functions $h_k(r), \xi_k(r)$ and $h(r), \xi(r)$ etc corresponding to $g_k$ and $g$.   Then for each $k\geq 1$ we have
\begin{enumerate}
\item $h_k(r)=(1/k) h(r/k)$
\item $\xi_k(r)=\xi(r/k)$
\item   The curvature of $g_k$ is bounded by a constant independent of $k$ because  $g_k$ is isometric to $g_0$.
\end{enumerate}
 Now
\be\label{e-comparison}\frac{h_k(r)}{h(r)}=\frac{1}{k}\frac{h(\frac r{k})}{h(r)}=\frac{1}{k}\exp( \int_{\frac r{k}}^r  \frac{\xi(s)}{s}ds) \ee
 and thus by {\bf(c1)} we have
\be\label{THMe2}
 e^{-\gamma}k^{(\alpha-1)}\leq \frac{h_k(r)}{h(r)} \leq e^{\gamma}k^{(\beta-1)}.
\ee
 By Remark \ref{r-comparison}, for any $k\ge 1$  we have
   \be\label{THMe3-2}
  e^{ -\gamma}k^{(1-  \beta  )} g_k \le g\le  e^{ \gamma}k^{(1-\alpha)}g_k.
   \ee
    Thus the Proposition follows in this case by the fact that $\beta<1$.

 Suppose $\xi$ satisfies {\bf(c2)}.  Let $\epsilon>0$ be given. Let $\widehat{g}$ be any $U(n)$-invariant non-negative bisectional curvature metric with $\widehat{h}(0)=1$ and generated by some $\widehat{\xi}$ with $\widehat{\xi}(r)=1$ for $r\geq 1$.  Let the curvature of $\widehat{g}$ be bounded by $\widehat{K}$.  For each $k\geq 1$ define the pullbacks   $\widehat g_k:=\phi_k^*(\widehat{g})$ as before.  Let $\widehat h_k(r),\widehat \xi_k(r)$ and $\widehat h(r)$ etc corresponding to $\widehat g_k$ and $ \widehat g$. Then properties (1) and (2) (3) above still hold,   but with $h, \xi, h_k,\xi_k$ replaced with $\widehat{h}, \widehat{\xi}, \widehat h_k, \widehat \xi_k$.

\vspace{10pt}

STEP 1: First note that  by \eqref{e-comparison} (applied to $\hat{h}(r)$) and the fact that $\hat{\xi}\leq 1$, we see that $h_k(r)$ is non-increasing in $k$ for all $r$. Now fix $\e>0$, by {\bf (c2)} there is $r_0>0$ such that if $r\ge r_0$, then for $k\ge 1$
\bee
\begin{split}
h(r)=&\widehat h(r) \frac{h(r)}{\widehat h(r)}\\
\ge & \widehat h_k(r)\exp\lf(\int_0^r\frac{\widehat\xi(s)-\xi(s)}sds\ri)   \\
\ge &\frac1\e \widehat h_k(r)
\end{split}
\eee
where we have used the fact that $\widehat \xi(r)=1$ for $r\ge 1$. On the other hand, by \eqref{e-comparison} one can see this is true for $r\le r_0$ if $k$ is large enough depending on $r_0$.  Hence by Remark \ref{r-comparison} we can find $k>1$  depending on $\e$ such that,
 $$
 \frac1\e g_k\le g
 $$
 on $\C^n$.

\vspace{10pt}

STEP 2: Define $$\widetilde{\xi}_k(r):=\xi_k(r)+ o_k(r)$$  where $o_k(r):[0, \infty)\to \R$ is a non-positive smooth function, to be chosen.
 where $o_k:[0, \infty)\to \R$ is   a smooth function  with $|o_k|\le \frac1k$, to be chosen. Let $\widetilde{h}_k(0)=1/k$ and consider the corresponding metric $\widetilde{g}_k$.

\vspace{10pt}
{\it Claim 1}: There exists a constant $R_k> k$ and a smooth function $o_k(r)$ which is $0$ on $[0, R_k]$ and satisfies:

 $$|o_k'(r)|\leq \frac{4}{kr}$$ and
\be\label{e-comparison-1}|\int_{R_k}^r \frac{\widetilde{\xi}_k(s)-\xi(s)}{s} ds|
=\lf|\int_{R_k}^r \frac{1+o_k(s)-\xi(s)}{s} ds\ri|
\leq 1+2\log2
\ee for $r\geq R_k.$

\vspace{10pt}

We may choose $R_k> k$ such that $1-1/k\leq \xi(r)\leq 1+1/k$ on $[R_k, \infty)$.   The construction of $o_k(r)$ follows from the construction in the proof of Proposition 5.1 \cite{ChauLiTam}.  We provide details here for the convenience of the reader.  We first choose a smooth non-increasing function $\rho(r):[0,\infty)\to \R$ with $\rho=0$ on $[0, 1]$, $\rho=1/k$ on $[2, \infty)$ and  $0\le \rho'\le 2/k$.  Let $$I(r):=\int_{2R_k}^r \frac{1+o_k(r)-\xi(s)}{s} ds.$$  (note that $\widetilde{\xi}_k(r)=1+o_k(r)$ for $r\geq R_k$).  For any positive sequence $\{r_i\}$ such that $r_0 :=R_k$ and $2r_i<r_{i+1}$,  define $o_k(r):=0$ if $r\in [0, r_0)$, $o_k(r):=\rho(r/r_0)$ if $r\in [r_0, 2r_0)$, and

\begin{equation}
o_k(r) : \begin{cases} =1/k &\mbox{if } r\in [2r_0, r_1]\\
  = 1/k-2\rho(r/r_1)&\mbox{if }  r\in [r_1, 2r_1]\\
=-1/k &\mbox{if } r\in [2r_1, r_2]\\
  = -1/k+2\rho(r/r_2)&\mbox{if }   r\in [r_2, 2r_2]\\
=1/k &\mbox{if }  r\in [2r_2, r_3]\\
  = 1/k-2\rho(r/r_3)&\mbox{if }  r \in [r_3, 2r_3]\\
...
 \end{cases}
\end{equation}
Now for $r\in [2r_0, \infty)$ we have $1-1/k\leq \xi(r) \leq 1+1/k$, and as long as $k\geq 2$ we have $|1+o_k(r)-\xi(r)|\leq 1$ as well.  The definition of $I(r)$ then gives the following for all $i\geq 0$

\be\nonumber\begin{split}  I(r)= I(2r_i)+\int_{2r_{i}}^{r} \frac{(1+(-1)^i /k)-\xi(s)}{s} ds \\ \end{split}\ee
for $r\in [2r_i, r_{i+1})$,
\be\nonumber\begin{split} I(r_{i+1})-\log 2\leq I(r)&\leq I(r_{i+1})+\log 2\\ \end{split}\ee for $r\in [r_{i+1}, 2r_{i+1})$.

By the fact $\xi(r)\to 1$, we may choose the $r_i$'s to be the smallest numbers with $I(r_1)=1$, $I(r_2)=-1$, $I(r_3)=1$,..etc, and the estimates above give
\be -1-\log 2\leq I(r)\leq 1+\log 2\ee
for all $r\in[2 R_k, \infty)$.  The integral bound in the claim follows from (2.7) and the fact that $|1+o_k(r)-\xi(r)|\leq 1$ for  $r\in [R_k, 2R_k]$.

Finally, we also have $|o'_k(r)|=\frac{2}{r_i}\rho'(\frac{r}{r_i})\leq \frac{4}{kr_i}\leq \frac{4}{kr}$ for all $r\in[r_i, 2r_i]$.

\vspace{10pt}
 {\it Claim 2}:  Let $o_k(r)$ be as in Claim 1.  Then $(1/4e\epsilon) \widetilde{g}_k \leq g \leq C_k  \widetilde{g}_k$ for some   $C_k$  and the curvature of $\widetilde{g}_k$ is bounded depending only on $\widehat{g}$.
\vspace{10pt}

%{\red Since $|o_k|\le \frac1k$,
To prove the first part of the claim,
when $r\leq R_k$ we have $\widetilde{g}_k(r)=g_k(r)$, and so we only have to consider when  $r\geq R_k$.  In this case,  we have $C_k \geq h(r)/\widetilde{h}_k(r)=(h(R_k)/\widetilde{h}_k(R_k))e^{\int_{R_k}^r \frac{1 + o_k(s)-\xi(s)}{s} ds}\geq \frac{1}{2e\epsilon}$   for some $C_k$ where we have used Step 1 and Claim 1.

 To prove the second part of the claim, note that  for $r\ge R_k$  we have $|\widetilde{\xi}_k'(r)|=|o_k'(r)|\leq 4/kr$ and

\be\begin{split}\widetilde{h}_k(r)&=\widetilde{h}_k(1)\exp\lf({-\int_1^r \frac{\widetilde{\xi}_k(s)}sds}\ri)\\
& =\widetilde{h}_k(1)\exp\lf({-\int_1^{R_k} \frac{\widetilde{\xi}_k(s)}sds
-\int_{R_k}^r\frac{ \widetilde{\xi}_k(s)}sds}\ri) \\
& \geq\widetilde{h}_k(1)\frac{1}{R_k} \exp\lf(-\int_{R_k}^r  \lf(\frac{\widetilde{\xi}_k(s)-\xi(s)}{s} +\frac{\xi(s)-1}{s} +\frac{1}{s} \ri)ds\ri) \\
& \geq\widetilde{h}_k(1)\frac{1}{R_k}\frac{ R_k}{4e^{1+\delta}r} \\
&=h_k(1)\frac{1}{4e^{1+\delta}r}\\
&\geq\frac{\widehat{h}(1)}{k}\frac{1}{4e^{1+\delta}r}\end{split}\ee
  where in the third line we have used that $0\leq \widetilde{\xi_k}\leq 1 $ by definition, and in the fourth line we have used Claim 1 and {\bf(c2).} so that $\int_{R_k}^r \frac{1-\xi(s)}{s}ds\ge -\delta$.

  Thus $|\widetilde{\xi}_k'(r)/\widetilde{h}_k(r)|\leq 16 e^{1+\delta}/\widehat{h}(1)$ for $r\ge R_k$. Since $\wt g_k(r)=g_k(r)$ for $r\le R_k$, by (3) we conclude that   the curvature of $\wt g_k$ is bounded by a constant independent of $k$ by \cite[Lemma 5.1(i)]{ChauLiTam}.

\end{proof}

Proposition \ref{s2p1} includes the case when $g$ has non-negative bisectional curvature and is not a cigar. In this case $\xi(r)\to \beta<1$ as $r\to\infty$ (see \cite{WZ}) where $\xi$ generates $g$ (see Appendix).  When $g$ is not a cigar, then the proposition is false in general.  In particular, we have

\begin{prop}\label{p-nonexitence} Let $g$ be a smooth $U(n)$-invariant metric generated by $\xi$ satisfying
\bee
\int_1^r\frac{1-\xi}sds\le c
\eee
for some $c$ for all $r\ge 1$. There exists a constant $c_1>0$ depending only on $g$ such that if $\widetilde g$   is  another $U(n)$-invariant \K metric   with bisectional  curvature bounded above by $1$  such that $g\ge \a \widetilde g$ for some $\a>0$, then $\a \le c_1$.
\end{prop}
\begin{proof}

Indeed, let $g$ be an $U(n)$-invariant complete \K metric generated by $\xi$ normalized by $h(0)=1$ such that
 \be\label{e-example-1}
  \int_1^r\frac{1-\xi}sds\le c
 \ee
 for some $c$ for all $r\ge1$.  We may assume $0\le \xi\le 1$ and $\xi=1$ for $r\ge 1$, for if $\hg$ is some metric generated by $\widehat\xi$ with $0\le \widehat\xi\le 1$ and $\widehat\xi=1$ for $r\ge 1$ and $\widehat h(0)=1$, then $\hg \ge c'g$ for some $c'>0$ by \eqref{e-example-1}.
\vspace{10pt}

 Let $\widetilde g$ be another $U(n)$-invariant \K metric such that $g\ge \a \widetilde g$ for some $\a>0$ with bisectional curvature bounded above by $1$. Assume $\widetilde g$ is generated by $\txi$. Then
\be\label{e-example-2}
 \limsup_{r\to\infty}\txi\ge 1.
\ee
 Otherwise, by \eqref{e-example-1}, we would have $$\frac{h(r)}{\widetilde h(r)}=\frac{h(0)}{\widetilde h(0)}\exp(\int_0^r\frac{\txi-\xi}sds)\to0
 $$ as $r\to\infty$.

 Let $0<\e<1$. By \eqref{e-example-2} and the fact that $\txi(0)=0$,  there is $r_1>0$ such that $\txi(r_1)=1-\e$. By the fact that $\txi(0)=0$, we can find $0\le r_0<r_1$ such that $\txi>0$ on $(r_0,r_1]$ and $\txi(r_0)=0$. It is easy to see that $\widetilde h(r)\le \widetilde h(r_0)$ for $r\in[r_0,r_1]$. Since the bisectional  curvature of $\widetilde{g}$ is bounded above by 1,
 $$
 \txi'(r)\le \widetilde h(r)\le \widetilde{h}(r_0)
 $$
 for all $r\in [r_0,r_1]$. Hence $\txi(r)\le \widetilde{h}(r_0)(r-r_0)$ for all $r\in [r_0,r_1]$. In particular, $1-\e=\txi(r_1)\le \widetilde{h}(r_0)(r_1-r_0)$ and so $r_2=r_0+(1-\e)/\widetilde{h}(r_0)\le r_1$.
 \be\label{e-example-3}
  \begin{split}
    h(r_2)   \ge &\a \widetilde h(r_2) \\
      =&\a\widetilde h(r_0)\exp(-\int_{r_0}^{r_2}\frac{\txi}sds)\\
\ge&\a\widetilde h(r_0)\exp(-\int_{r_0}^{r_1}\widetilde{h}(r_0)ds)\\
      \ge & \a \widetilde{h}(r_0)\exp\lf(   \widetilde{h}(r_0)  (r_1 -r_0)\ri)\\
      \ge&\a \widetilde{h}(r_0)\exp(-(1-\e))
  \end{split}
  \ee
  On the other hand, by the definition of $r_2$ we have $\widetilde{h}(r_0)r_2 \ge (1-\e)$ and thus
  $$
  h(r_2)=h(1)\exp(-\int_1^{r_2}\frac{\xi}sds)=\frac{h(1)}{r_2}\le \frac{\widetilde{h}(r_0)h(1)}{1-\e}.
  $$
  By \eqref{e-example-3}, we conclude that
  \bee
  \frac{\widetilde{h}(r_0)h(1)}{1-\e}\ge \a \widetilde{h}(r_0)\exp(-(1-\e)).
  \eee
  Since $\e$ is arbitrary, $\a \le h(1)$.

 \end{proof}

\begin{proof}[Proof of Theorem \ref{t-Un-longtime-2} (i) and (ii)]

 First of all, uniqueness in each case (i), (ii) and (iii) follows from Theorem \ref{t-uniqueness-2}.  Indeed, let $g_1(t), g_2(t)$ be two solutions as in the Theorem.  Then using the bounded curvature metric $\widehat{g}=g_1(1)$ in Theorem \ref{t-uniqueness-2}, and noting the solutions $g_1(t), g_2(t)$ are both equivalent to $g_0$ and hence also $\widehat{g}$, we conclude by Theorem \ref{t-uniqueness-2} that $g_1(t)=g_2(t)$ for all $t$.

Now consider case (i).  Suppose $\xi$ satisfies {\bf (c1)} or {\bf (c2)}.
 By \cite[Theorems 5.4, 4.2]{ChauLiTam} there exists a smooth $U(n)$-invariant solution $g(t)$ to \KRF on $\C^n\times[0, T]$ for some $T>0$ such that $g(0)=g_0$, the curvature of $g(t)$ is uniformly bounded by $c/t$  and $c^{-1}g(t)\le g_0\le cg(t)$ for some $c>0$ for all $t\in (0, T]$. Fixed $0<t_0<T$. By Proposition \ref{s2p1} and the fact that $g(t_0)$ is uniformly equivalent to $g_0$,  for any $\e>0$, there is a complete $U(n)$-invariant metric $\hg_\e$ with curvature bounded by $K$ with $K$ being independent of $\e$ such that
 $$
\frac1\e\hg_\e\le g(t_0)\le c(\e)\hg_\e
$$
for some constant $c(\e)$ which may depend on $\e$. By Lemma \ref{l-existencetimeestimate-1} and Theorem \ref{t-uniqueness-2}, the \KR flow $g(t)$   can  be extended   to $[0,T_\e)$,  where $T_\e=\frac1{2nK\e}$, such that $g(t)$ is uniformly  equivalent to $g_0$ for all $t\in[0,T')$ for any $T'<T_\e$ and has uniformly bounded curvature on $(\delta,T')$ for all $0<\delta<T'<T_\e$. Let $\e\to0$, one may conclude the theorem is true.

 Now consider case (ii).  Here we also have a short time solution $g(t)$ on $\C^n\times[0, T]$ as in case (i) by \cite{Shi2}.  By the equivalence condition in (ii), it is easy to see that we can apply the exact same argument as above to conclude the Theorem is true in this case as well.

\end{proof}

\begin{rem}\label{r-Un-longtime-2} It is easy to see that if $\xi$ satisfies:
$$
-c\le\int_1^r\frac{\xi(s) -a}s ds\le c
$$
for some $c>0$ and $0<a<1$ for all $r\ge 1$, then $\xi$ satisfies {\bf(c1)}. This generalizes  \cite[Theorem 5.5]{ChauLiTam} in case $a<1$.
\end{rem}

If $\xi$ satisfies {\bf(c3)} in Theorem \ref{t-Un-longtime-2}, then the previous argument does not work in light of Proposition \ref{p-nonexitence}.  However, if $g_0$ also has nonnegative bisectional curvature (and is thus a cigar) we may use other methods.   Before we prove the theorem in this case we need a more general form of \cite[Theorem 2.1]{NiTam2004}, which will also be used later.

\begin{lem}\label{l-F-bound} Let $(M^n,g(t))$ be complete noncompact solution of the \KR flow \eqref{krf} on $M\times[0,T)$ with bounded nonnegative bisectional curvature. Let
$$
F(x,t)=\log\lf(\frac{\det(g_\ijb(x,t))}{\det(g_\ijb(x,0))}\ri)
$$
and for any $\rho>0$, let $\mathfrak{m}(\rho,x,t)=\inf_{y\in B_0(x,\rho)}F(y,t)$. Then there is $c>0$ depending only on $n$ such that for any $x_0\in M$ and for all $\rho, t>0$
\be\label{e-F-bound-1}
\begin{split}
-F(x_0,t)\leq &c\bigg[\lf(1+\frac{t(1-\mathfrak{m}(\rho,x_0,t))}{\rho^2}\ri)\int_0^{2\rho}sk(x_0,s)ds\\
&-
\frac{t\mathfrak{m}(\rho,x_0,t))(1-\mathfrak{m}(\rho,x_0,t))}{\rho^2}\bigg]\end{split}
\ee
where $B_0(x,\rho)$ is the geodesic ball with respect to $g_0=g(0)$, and $k(x,s)$ is the average of the scalar curvature $\mathcal{R}_0$ of $g_0$ over $B_0(x,s)$.
\end{lem}
\begin{proof} By \cite[(2.6)]{NiTam2004}, if $G_\rho$ is the positive Green's function on $B_0(x_0,\rho)$ with Dirichlet boundary value, then
\be\label{e-F-bound-2}
\begin{split}
&\qquad\int_{B_0(x_0,\rho)}G_\rho(x_0,y)\lf(1-e^{F(y,t)} \ri)dv_0\\
&\le t\int_{B_0(x_0,\rho)}G_\rho(x_0,y)\mathcal{R}_0(y)dv_0+\int_0^t\int_{B_0(x_0,\rho)}G_\rho(x_0,y)\Delta_0 (-F(y,s)) dv_0ds\\
&=:I+II.
\end{split}
\ee
As in \cite[p.126]{NiTam2004},
\be\label{e-F-bound-3}
II\le -t\mathfrak{m}(\rho,x_0,t).
\ee
As in \cite[(2.8)]{NiTam2004} there is a constant $c_1$ depending only on $n$ such that
\bee
\begin{split}
\rho^2&\aint_{B_0(x_0,\frac15\rho)}(-F(y,t))dv_0\\
&\le c_1t(1-\mathfrak{m}(\rho,x_0,t))\lf(\int_{B_0(x_0,\rho)}G_\rho(x_0,y)\mathcal{R}_0(y)dv_0
-\mathfrak{m}(\rho,x_0,t)\ri).\end{split}
\eee
Using the fact that $\Delta_0(-F)\ge -\mathcal{R}_0$ and \cite[Lemma 2.1]{NiTam2004}, we obtain
\be\label{e-F-bound-4}
\begin{split}
-&F(x_0,t)\le\\
& \int_{B_0(x_0,\frac15\rho)} G_\rho(x_0,y)\mathcal{R}_0(y)dv_0\\&+
c_2\rho^{-2}t(1-\mathfrak{m}(\rho,x_0,t))\lf(\int_{B_0(x_0,\rho)}G_\rho(x_0,y)\mathcal{R}_0(y)dv_0
-\mathfrak{m}(\rho,x_0,t)\ri)
\end{split}
\ee
for some $c_2$ depending only on $n$. As in \cite[p.127]{NiTam2004}, we get the result.
\end{proof}

We also need the following:

\begin{lem}\label{l-limit-Un-5} Let $g(t)$ be the complete $U(n)$-invariant solution of the \KR flow \eqref{krf} with nonnegative bisectional curvature.
Let
$$
F(r,t):=F(z,t)=\log\lf(\frac{\det(g_\ijb(z,t))}{\det(g_\ijb(z,0))}\ri)
$$ where $r=|z|^2$. Then for $r\ge1$ and for all $t$
\bee
F(r,t)\ge -c-n\log r+F(1,t)
\eee
for some constant $c>0$ depending only on $g(0)$. If in addition the generating function $\xi$ of $g_0$ satisfies:
\be\label{l-limit-Un-5e1}
\lim_{r\to\infty}\int_1^r\frac{1-\xi(s)}sds=b<\infty,
\ee
then for $r\ge1$ and for all $t$
\bee
F(r,t)\ge -c+nF(1,t)
\eee
for some constant $c>0$ depending only on $g(0)$.
\end{lem}
\begin{proof}    Consider the functions $\xi(r, t), h(r, t), f(r, t)$ corresponding $g(r, t)$.
  Then $0\leq\xi(r, t)\leq 1$ since $g(r,t)$ has non-negative bisectional curvature, and by Theorem \ref{appendixthm} we then get $0\leq h(r, t), f(r, t)\leq 1$.  Thus for $r\ge1$ we have
$$
f(r,t)=\frac1r\int_0^rh(s,t)ds\geq\frac1r \int_0^1h(s,t)dt= \frac1rf(1,t).
$$

$$
h(r ,t)=h(1,t)\exp\lf(\int_1^r-\frac{\xi(s)}sds\ri)\ge \frac1r h(1,t),
$$
and using the formula $\det(g_\ijb(r,t))=h(r,t)f^{n-1}(r,t)$ we then get

\bee
\begin{split}
%\frac{\det(g_\ijb(r,t))}{\det(g_\ijb(r,0))}\ge &\frac{\det(g_\ijb(r,t))}{\det(\widehat g_\ijb(r)}\\
\frac{\det(g_\ijb(r,t))}{\det(g_\ijb(r,0))}\ge&\det(g_\ijb(r,t)) \\
\ge &\frac1{r^n}h(1,t)f^{n-1}(1,t)\\
=&\frac1{r^n}\det(g_\ijb(1,t))\\
=&\frac1{r^n} \frac{\det(g_\ijb(1,t))}{\det(g_\ijb(1,0))}\cdot  \det(g_\ijb(1,0)).
\end{split}
\eee
From this, it is easy to see the first result follows.  Now suppose the generating function $\xi$ of $g_0$ also satisfies   \eqref{l-limit-Un-5e1}.  Then for $r\ge1$,
 \bee
 \begin{split}
 \frac{h(r,t)}{h(r,0)}=&\frac{h(1,t)\exp\lf(-\int_1^r\frac{\xi(s,t)}sds\ri)}{h(1,0)
 \exp\lf(-\int_1^r\frac{\xi(s,0)}sds\ri)}\\
 \ge&\frac{h(1,t)}{h(1,0)}\exp\lf(\int_1^r\frac{\xi(s,0)-1}sds\ri)\\
 \ge &c_1\frac{h(1,t)}{h(1,0)}
 \end{split}
 \eee
 for some constant $c_1>0$ independent of $r,t$, provided $r\ge1$.  Also for $r\ge1$
\bee
 \begin{split}
 \frac{f(r,t)}{f(r,0)}=&\frac{\int_0^rh(s,t)ds}{\int_0^rh(s,0)ds}\\
 = &\frac{\int_0^1h(s,t)ds+\int_1^rh(s,t)ds}{\int_0^1h(s,0)ds+\int_1^rh(s,0)ds}\\
 \ge&\frac{h(1,t)+ c_1\frac{h(1,t)}{h(1,0)}\int_1^rh(s,0)ds}{\int_0^1h(s,0)ds+\int_1^rh(s,0)ds}\\
 \ge&c_2\frac{h(1,t)}{h(1,0)}
 \end{split}
 \eee
for some constant $c_2>0$ independent of $r,t$, provided $r\ge1$.

Hence for $r\ge1$, at the point $|z|^2=r$
\bee
\begin{split}
\frac{\det(g_\ijb(z,t))}{\det(g_\ijb(z,0))}=&\frac{h(r,t)f^{n-1}(r,t)}{h(r,0)f^{n-1}(r,0)}\\
\ge&c_3\lf(\frac{h(1,t)}{h(1,0)}\ri)^n\\
\ge&c_3\lf(\frac{h(1,t)f^{n-1}(1,t)}{h(1,0)f^{n-1}(1,0)}\ri)^n\\
=&c_3\lf(\frac{\det(g_\ijb(1,t))}{\det(g_\ijb(1,0))}\ri)^n
\end{split}
\eee
for some $c_3>0$, where we have used the fact that $f(r,t)\le f(r,0)$.  From this the second result follows.
\end{proof}

\begin{proof}[Proof of Theorem \ref{t-Un-longtime-2} (iii)]  Let $g_0$ be generated by $\xi$ normalized so that $h(0)=1$  and so that  {\bf(c3)}  is satisfied.   We may assume that $\xi$ is not identically zero, otherwise $g_0$ would be the standard metric and the theorem is obviously true.

By \cite[Theoerm 5.4]{ChauLiTam} and \cite{YZ}, the \KR flow has a $U(n)$-invariant complete solution $g(t)$ on $M\times[0,T)$ which is equivalent to $g_0$ and has bounded non-negative bisectional curvature.  Hence we may assume that $g_0$ has bounded curvature.  Let $T$ be the maximal such existence time.

Now, let $\mathfrak{m}(t)=\inf_{z\in \C^n}F(z,t)$. By Lemma \ref{l-limit-Un-5} and the fact $F\le0$,

\bee
\mathfrak{m}(t)\ge -c_1+n\inf_{\{z\in\C^n|\ |z|^2\le1\}}F(z,t)
\eee
for some $c_1>0$ depending only on $g_0$, and so
Since $F\le0$.
\be\label{e-m-1}
  -\mathfrak{m}(t)\le  c_1-n\inf_{\{z\in \C^n, |z|^2\le 1\}}F(z,t).
\ee
On the other hand by the proof of \cite[Theorem 7]{WZ}, there is a constant $c_2>0$ such that for all $r>0$
\bee
\frac1{V_0(0,\rho)}\int_{B_0(0,\rho)}\mathcal{R}(0)\le \frac {c_2}{1+\rho}
\eee
where $B_0$ is the geodesic ball on with respect to $g(0)$, where $\mathcal{R}(0)$  is the scalar curvature of $g_0$. Hence by volume comparison, there is a constant $c_3>0$ such that
\bee
\frac1{V_0(z,\rho)}\int_{B_0(z,\rho)}\mathcal{R}(0)\le \frac {c_3}{1+\rho}
\eee
for all $r>0$ and for all $z$ with $|z|^2\le 1$.   By  Lemma \ref{l-F-bound}, \eqref{e-m-1}
\bee
-\mathfrak{m}(t)\le c_1+nc_4\lf[\lf(1+\frac{t(1-\mathfrak{m}(t))}{\rho^2}\ri)c_5\rho
-\frac{t\mathfrak{m}(t)(1-\mathfrak{m}(t))}
{\rho^2}\ri]
\eee
for some constants $c_4, c_5>0$ independent of $t, \rho$.
Choose $\rho^2=\frac1{2nc_1}t(1-\mathfrak{m}(t))$. Then one can conclude that for all $0<t<T<\infty$
\bee
\frac12(1-\mathfrak{m}(t) )\le c_6+ c_7T^\frac12(1-\mathfrak{m}(t))^\frac12
\eee
for some constants $c_6,  c_7>0$ independent of $T$.
 So $(1-\mathfrak{m}(t) )$ and hence $-\mathfrak{m}(t)$ is uniformly bounded on $[0,T)$.
Then one can proceed as in \cite{Shi2}  to conclude that the theorem is true.
\end{proof}

The following Theorem gives longtime existence under more general conditions, only in this case we do not know if the curvature of $g(t)$ is necessarily bounded for each time.

\begin{thm}\label{t-Un-longtime-2-2} Let $g_0$ be a complete $U(n)$-invariant \K metric on $\C^n$  generated by $\xi$.  Suppose
   \be\label{e-longtimecond-1}
   \lim_{r\to\infty}\int_1^r\frac{1-\xi(s)}sds=\infty.
   \ee
   Then the \KR flow \eqref{krf} with initial data $g_0$ has smooth complete long time $U(n)$-invariant solution $g(t)$.
   \end{thm}
   \begin{proof}

 By \eqref{e-longtimecond-1} we may choose some increasing sequence $r_k\to \infty$ such that $\xi(r_k)<1$.  For each $k$, let $\xi_k$ be smooth such that $\xi_k(r)=\xi(r)$ for $r\le r_k-\delta_k$ for some $0<\delta_k<1$  and $\xi_k(r)=\xi(r_k)$ for $r\ge r_k$ such that
 \be\label{t-Un-longtime-2-2e0}
 \exp\lf(\int_{r_k-\delta_k}^{r_k}\lf|\frac{\xi(s)-\xi_k(s)}{s}\ri|ds\ri)\le 2.
 \ee
 Let $g_k$ be the $U(n)$-invariant \K metric associated to $\xi_k$ and consider also the corresponding functions $h_k, f_k$. Then $g_k$ has bounded curvature. By Theorem \ref{t-Un-longtime-2} or by \cite[Theorem 5.5]{ChauLiTam}, for each $k$ the \KR flow has a longtime solution $g_k(t)$ on $\C^n\times[0, \infty)$ with bounded curvature.

 Now let $\epsilon>0$ be given and consider the sequence $\widehat{g}_i =\phi_i^* \widehat{g}$ as in the first part of the proof of Proposition \ref{s2p1}.  Then as in Step 1 there, we may fix $i_0$ sufficiently large so that

 \be\label{1}
 \widehat{g}_{i_0}(r) \leq \e g(r)
 \ee
 for all $r$.  In particular, for every $k$ we have \be\label{t-Un-longtime-2-2e1}\widehat{g}_{i_0}(r) \leq 2\e g_k(r)\ee for $r\leq r_k+\delta_k$ by \eqref{t-Un-longtime-2-2e0} and the fact that $g_k(r)=g(r)$ for $r\leq r_k$.  On the other hand, for some $k_0$ we have $\widehat{\xi}_{i_0}(r)=1$ for $r\geq r_{k_0}$ and thus for all $k\geq k_0$, $$\frac{\widehat{h}_{i_0}(r)}{h_k(r)}=\frac{\widehat{h}_{i_0}(r_k +\delta )}{h_k(r_k  +\delta )}\exp\lf(\int_{r_k+\delta_k}^{r}\frac{\xi_k(s)-1}{s}ds\ri)\leq 2 \e$$

 In particular, \eqref{t-Un-longtime-2-2e1} holds for all $r\in[0, \infty)$ while the curvature of $\widehat{g}_{i_0}$ is bounded by $\widehat{K}$.  By \cite[Lemma 3.1]{ChauLiTam}, we have
\be
g_k(r,t)\ge \lf(\frac1n-2n\widehat{K}\e t\ri)\widehat{g}_{i_0}(r)
\ee
on $\C^n \times[0, \frac{1}{4n^2\widehat{K}\e})$ for all $k\geq k_0$. Note that $g_k(x,0)\to g_0(x)$ uniformly on compact sets. By the proof of \cite[Theorem 4.2]{ChauLiTam}, $g_k(x,t)$ converges subsequentially uniformly on compact sets of $M\times[0,\frac{1}{4n^2\widehat{K}\e})$, to a smooth solution of \KR flow \eqref{krf}.
Since $\epsilon>0$ was arbitrarily chose, we conclude that the theorem is true by a diagonal process.
\end{proof}

\begin{rem}
If $\xi(r)\le 0$ near infinity in Theorem \ref{t-Un-longtime-2-2}, then the result follows from
 \cite[Th. 5.3]{ChauLiTam}.
\end{rem}

In the above theorem we do not assume that $g_0$ has bounded curvature. However, even if $g_0$ has bounded curvature the solution $g(t)$ in the above theorem may not have bounded curvature for $t>0$, unlike the result in Theorem \ref{t-Un-longtime-2}.

\section{Convergence}

By Corollary \ref{c-nonnegative-longtime}, for any $U(n)$-invariant complete \K metric $g_0$ with nonnegative bisectional curvature, the \KR flow with initial data $g_0$ has a longtime $U(n)$-invariant solution equivalent to $g_0$ with bounded nonnegative bisectional curvature.  In this section we discuss the convergence of such solutions.  Again, we will make extensive use of the notation and results in Theorem \ref{appendixthm} in the appendix on $U(n)$-invariant \K metrics $g$, their associated auxiliary functions $\xi, h_\xi, f_\xi$, and the associated components $A, B, C$ of the curvature tensor.  We also recall the notation in definition \ref{definition1}.

%\subsection{Non-negatively curved $U(n)$-invariant solutions}

\begin{thm}\label{t-Un-limit-1} Let $g(t)$ be a complete longtime $U(n)$-invariant solution of the \KR flow \eqref{krf} with bounded non-negative bisectional curvature, and assume $g(0)$ also has bounded curvature. Then $g(t)$ converges, after rescaling at the origin, to the standard Euclidean metric on $\C^n$.
\end{thm}

In order to prove the theorem, we need the following lemmas.

\begin{lem}\label{l-limit-Un-1}  Let $g(t)$ be as in Theorem \ref{t-Un-limit-1}.  Suppose the curvature of $g(t)$ is uniformly bounded by $c_1$, in $D(R)\times[0,\infty)$, where $D(R)=\{|z|^2<R\}$. Then there is a constant $c_2$ depending only on $c_1$ and  $R$ such that
 \be\label{e-metricbound-1}
 h(r,t)\le h(0,t)\le c_2 h(r,t);\ \  f(r,t)\le f(0,t)\le c_2 f(r,t)
 \ee
 for all $0<r<R$ and for all $t$.
 \end{lem}
 \begin{proof} By Remark \ref{r-comparison}  and the fact that $g(t)$ has nonnegative bisectional curvature, we have $h(0,t)\ge h(r,t)$ and $f(0,t)\ge f(r,t)$ for all $r>0$.  On the other hand, we have $A(r,t), B(r,t), C(r,t)\le c_1$ in $D(R)\times[0,\infty)$ by hypothesis.  Thus $C=-\frac{2f_r}{f^2}$ gives
 $$
 \frac1{f(r,t)}-\frac{1}{f(0,t)}\le \frac {c_1R}2
 $$
 in $D(R)\times[0,\infty)$, and by $f(0,t)\le f(0,0)=h(0,0)=1$ we get
$$
 f(0,t)\le \lf(\frac {c_1R}2+1\ri) f(r,t).
 $$
Also, $A= \frac{\xi_r(r,t)}{h(r,t)}\le c_1$ gives
 $$
 \xi_r(r,t)\le c_1h(r,t)\le c_1h(r,0)\le c_1 h(0,0)=c_1
$$
in $D(R)\times[0,\infty)$, and thus $\xi(r,t)\le c_1r$, giving
$$
h(r,t)=h(0,t)\exp\lf(\int_0^r\frac{-\xi(s,t)}s ds\ri)\ge \exp(-c_1R)h(0,t).
$$
This completes the proof of the lemma.
 \end{proof}

 \begin{lem}\label{l-limit-Un-2} Let $g(t)$ be as in Theorem \ref{t-Un-limit-1}. Then for any $R>0$
 the curvature of $g(t)$ is uniformly bounded in $D(R)\times[0,\infty)$.
 \end{lem}
 \begin{proof}  Let
 $$
 k(z,\rho)=\frac{1}{V_0(z,\rho)}\int_{B_0(z,\rho)}\mathcal{R}(0)
 $$
 be the average of the scalar curvature $\mathcal{R}(0)$ of $g_0$ over the geodesic ball
 $B_0(z,\rho)$ with respect to $g_0$. Let
 $$
 k(\rho)=\sup_{|z|\le 1}k(z,\rho).
 $$
 By \cite[Theorem 7]{WZ}, there is a constant $c_1$ such that
 \be\label{e-curvaturebound-1}
 k(\rho)\le \frac{c_1}{1+\rho}.
 \ee
 Suppose $|z|^2=r$, then the distance $\rho(z)=\rho(r)$ from $z$ to the origin satisfies
 \be\label{e-curvaturebound-2}
 \rho(z)=\rho(r)=\frac12 \int_0^r\frac{\sqrt h}{\sqrt s}ds\ge c_2\log r
 \ee
 for some constant $c_2>0$ for all $r\ge 1$. Let

 $$
 F(z,t)=\log\frac{\det(g_\ijb(z,t))}{\det(g_\ijb(z,0))}
 $$
 and let
 $\mathfrak{m}(\rho, t)=\inf_{z\in\C^n, \rho(z)\le \rho}F(z,t)$.
 Fix $r_0>1$ and let $\rho_0=\rho(r_0)$. Denote $-\mathfrak{m}(\rho_0, t)$ by $\eta(t)$. By Lemmas \ref{l-F-bound} and   \ref{l-limit-Un-5}, there exist positive constants $c_3, c_4$ independent of $t$ and $\rho$ such that
 \bee
\begin{split}
 \eta(t) \le&c_4\lf[\lf(1+\frac{t(1-\mathfrak{m}(\rho+\rho_0,t))}{\rho^2}\ri)
K(\rho)-\frac{t\mathfrak{m}(\rho+\rho_0,t)(1-\mathfrak{m}(\rho+\rho_0,t))}{\rho^2} \ri]\\
\le &c_4\bigg[\lf(1+\frac{t(1+c_1+\log \widetilde r(\rho)+\eta(t))}{\rho^2}\ri)
K(\rho)\\
&+\frac{t(c_3+\log \widetilde r(\rho)+\eta(t))(1+c_1+\log \widetilde r(\rho)+\eta(t)}{\rho^2} \bigg]
\end{split}
\eee
where
$$
K(\rho)= \int_0^{2\rho}sk(s)ds
$$
and  $\widetilde r(\rho)$ is such that,
$$
\rho+\rho_0=\frac12\int_0^{\widetilde r(\rho)}\frac{\sqrt h}{\sqrt s} ds.
$$
By \eqref{e-curvaturebound-1} and \eqref{e-curvaturebound-2},  there is a constant $c_5$ independent of $\rho$ and $t$, such that

\bee
1+\eta(t)\le c_5\lf(\rho+\frac{t(1+\eta(t))}\rho+\frac{t(1+\eta(t))^2}{\rho^2}+t\ri).
\eee
Let $\rho^2=2c_5t(1+\eta(t))$, then
\bee
1+\eta(t)\le c_6\lf(t^\frac12(1+\eta(t))^\frac12+t\ri)
\eee
for some constant $c_6>0$ independent of $t$.
From this we conclude that $\eta(t)\le c_7(1+t)$ for some constant $c_7$ independent of $t$. Hence

$$
-F(z,t)\le c_7(1+t)
$$
for all $z$ with $|z|^2\le r_0$,
which implies
$$
 \int_t^{2t}\mathcal{R}(z,s)ds\le\int_0^{2t}\mathcal{R}(z,s)ds=-F(z,t)\le c_7(1+t).
 $$
 On the other hand, by the Li-Yau-Hamilton inequality \cite{Cao-1992}, $s \mathcal{R}(z,s)\ge t\mathcal{R}(z,t)$ for $s\ge t$. Hence we have
 $$
 \mathcal{R}(z,t)\le c_8
 $$
for all $t$ and for all $z$ with $|z|^2\le r_0$. This completes the proof of the lemma.
\end{proof}

\begin{proof}[Proof of Theorem \ref{t-Un-limit-1}]
Let $a(t)=h(0,t)$.  We claim that the curvature of $a(t)^{-1}g(z,t)$ converges to 0 uniformly on compact sets. Note that $a(t)^{-1}g(x,t)$ has nonnegative bisectional curvature. Let $\mathcal{R}(z,t)$ be the scalar curvature of $g(t)$ at $z\in \C^n$. Suppose first that $\lim_{t\to\infty}\mathcal{R}(0,t)=0$. Then by the Li-Yau-Hamilton inequality \cite{Cao-1992}, we conclude that $\lim_{t\to\infty}\mathcal{R}(z,t)=0$ uniform on compact sets. Since $a(t)\le a(0)=h(0,0)=1$, the claim is true in this case.

Suppose on the other hand that there exist $k\to\infty$ and $c_1>0$ such that $\mathcal{R}(0,t_k)\ge c_1$ for all $k$. We may assume that $t_{k+1}\ge t_k+1$. By the Li-Yau-Hamilton inequality again, there is $c_2>0$ such that $\mathcal{R}(0,t_k+s)\ge c_2$ for all $k$ and for all $0\le s\le 1$. Since the Ricci tensor of $g(t)$ at the origin is $\Ric=\frac{\mathcal{R}}ng$, using the \KR flow equation, we have
$$
h(0,  t_{k+1}) \le h(0,t_k+1)\le e^{-c_3}h(0,t_k)
$$
for some $c_3>0$ for all $k$. Hence $h(0,t_k)\to 0$ as $k\to\infty$. Since $h(0,t)$ is nonincreasing, we have $\lim_{t\to\infty}a(t)=\lim_{t\to\infty}h(0,t)=0$. On the other hand, the curvature of $g(t)$ is uniformly bounded on compact sets by Lemma \ref{l-limit-Un-2}. Thus our claim is true in this case as well.

 Consider any sequence $t_k \to \infty$.  Let $a_k=h(0,t_k)$ and let $\widetilde g_k(x,t)=\frac{1}{a_k }g(x, a_kt+t_k).$
 Then $\widetilde g_k(t)$ is a $U(n)$-invariant solution to the \KR flow on $\C^n\times[-\frac{t_k}{a_k}, \infty)$. Note that $-t_k/a_k\le -t_k$ because $a_k\le 1$. By Lemmas \ref{e-metricbound-1}, \ref{l-limit-Un-2}, for any $R>0$, $\widetilde g_k(x,0)$ is uniformly equivalent to the standard Euclidean metric $g_e$ on $D(R)$ (with respect to $k$). By the claim above and the Li-Yau-Hamilton inequality \cite{Cao-1992},  the curvature of the metrics $\widetilde{g}_k(x,t)$ approach zero uniformly  (with respect to  $k$)  on compact subsets of $\C^n\times(-\infty, 0]$.

   In particular, we conclude that $\widetilde g_k(t)$ is uniformly equivalent to $g_e$ in $D(R)$ provided $-1\le t\le 0$, and thus by \cite[Theorem 2.2]{ChauLiTam}, we have for any $m\ge0$, there is a $c_4$ depending on $R$ such that
 $$
 |\nabla_e^m\widetilde g_k(0)|\le c_4
 $$
 on $D(\frac R2)$, where $\nabla_e$ is the derivative with respect to the standard Euclidean metric.  From this it is easy to conclude the subsequence convergence of $\widetilde g_k(0)$ uniformly and smoothly on compact subsets of $\C^n$ to a flat $U(n)$-invariant \K  metric $g_\infty$, generated by some $\xi_\infty$ say.  Since the curvature is zero, we have $\xi'_\infty\equiv0$ and thus $\xi_\infty\equiv 0$. Moreover, at the origin $(g_\infty)_\ijb=\delta_{ij}$. Hence $h_\infty(0)=1$ which implies that $(g_\infty)_\ijb=\delta_{ij}$ everywhere.  From this the Theorem follows   as $t_k$ was chosen arbitrarily.

 \end{proof}

In some cases, we may remove the assumption that the metric is $U(n)$-invariant.

\begin{prop}\label{p-limit-1} Let $g_0$ be a $U(n)$-invariant complete \K metric on $\C^n$ with nonnegative bisectional curvature with maximum volume growth (i.e. $\xi\to a<1$). Let $G_0$ be another complete \K metric on $\C^n$ with bounded and nonnegative bisectional curvature. Suppose $G_0$ is uniformly equivalent to $g_0$. Let $G(t)$ be the long time solution of the \KR flow obtained in \cite{Shi2}. Then for any $t_k\to \infty$ and any fixed $v\in T^{1,0}(\C^n)$ at the origin with $v\neq 0$, there is a subsequence still denoted by $t_k$ such that
   $$
   \frac1{G(t_k)(v,\bar v)}G(t_k)
   $$
    will converge uniformly and smoothly on compact set to a complete flat metric on $\C^n$.

\end{prop}

\begin{proof}
By the proof of Theorem \ref{t-Un-limit-1}, for and $r>0$, there is a constant $c_r$ such that
   $$
    c_r^{-1} g_e\le \frac1{g(t_k)(v,\bar v)}g(t_k)\le c_r g_e
    $$
    on $|z|^2<r$, where $g_e$ is the standard metric on $\C^n$.  By Lemma \ref{l-limit-comparison-1}, we see that
   $$
    (c_1c_r)^{-1} g_e\le \frac1{G(t_k)(v,\bar v)}G(t_k)\le c_1c_r g_e
    $$
    on $|z|^2<r$, for some $c_1>0$ independent of $k$ and $r$.   Also, by the results in \cite{Shi2}, the curvature of $G(t)$ is uniformly bounded on $\C^n \times [0, \infty)$. The argument for convergence is now similar to that in the proof of Theorem \ref{t-Un-limit-1}.

\end{proof}

\begin{lem}\label{l-limit-comparison-1} Let $g_0$ and $G_0$ be complete \K metrics on $M$ with bounded nonnegative bisectional curvature and maximum volume growth.  Let $g(t), G(t)$ be the long time solutions of the \KR flow \eqref{krf} with initial data $g_0, G_0$ respectively as obtained in \cite{Shi2}.  If $g_0$ and $G_0$ are equivalent, then there exist $c>0$ such that for all $t>0$
  $$
  c g(t)\le G(t)\le c^{-1} g(t).
  $$

  \end{lem}
  \begin{proof} Let $\mathcal{R}_g(x,t)$, and $\mathcal{R}_G(x,t)$ be the scalar curvatures of $g(t)$ and $G(t)$ at $x$ respectively. By a result of Ni, $\mathcal{R}_g(\cdot,0)$ and $\mathcal{R}_G(\cdot,0)$ decay like $r^{-2}$ on average uniformly at all points, and by \cite{Shi2}, $g(t)$ and $G(t)$ exist for all time and there is $c_1>0$ such that for all $x, t$
  \be\label{e-limit-comparison-1}
  t\mathcal{R}_g(x,t),\  t\mathcal{R}_G(x,t)\le c_1.
  \ee
  By assumption, there is $c_2>0$ such that
  $$
  c_2g_0\le G(0)\le c_2^{-1}g_0.
  $$
 Since $g(t)$ and $G(t)$ is nonincreasing, we have
 $$
 G(0)\ge c_2 g(t),\  g(0)\ge c_2 G(t),
 $$
 for all $t$. Fix $t_0>0$. The bisectional curvature of $g(t_0)$, $G(t_0)$ are   bounded above by $c_3t_0^{-1}$ for some $c_3$ independent of $t_0$. By \cite[Lemma 3.1]{ChauLiTam}, there exists a constant $c_4$ independent of $t_0$, $t$ such that for all $t>0$,
 $$
 G(t)\ge c_2\lf(\frac 1n-c_4t_0^{-1}t\ri)g(t_0),
 $$
 and hence
 $$
 G(\frac1{2nc_4}t_0)\ge \frac{c_2}{2n}g(t_0).
 $$
 On the other hand, by the \KR flow equation and \eqref{e-limit-comparison-1}, we have
 $$
 G(t_0)\ge c_5G(\frac1{2nc_4}t_0)
 $$
 for some $c_5>0$ independent of $t_0$. Hence we have
 $$
 G(t_0)\ge c_6 g(t_0)
 $$
  for some $c_6>0$ independent of $t_0$. Similarly, one can also prove that $g(t_0)\ge cG(t_0)$ for some $c>0$ independent of $t_0$. From this the lemma follows.

   \end{proof}

%\section{Unresolved Questions}
%\begin{enumerate}
 %\item In light of Theorem \ref{t-uniqueness-2}, the solution in Theorem. 4.2 in \cite{ChauLiTam} does not depend on the subsequence.  Is it true that the solution in Theorem. 4.1 in \cite{ChauLiTam} does not depend on the subsequence?
%\item Can we still have the conclusion in Theorem \ref{t-Un-longtime-2} if: we do not assume non-negative bisectional in the case of $(\mathbf{c3})$, or more generally if we let $\beta=1$ in $(\mathbf{c1})$?

%\item What can we say if $a=1$ in Proposition \ref{propgeneralmetricslongtime}
%?  Can we remove the assumption of bounded curvature here?

%\item What can we say if $g_0$ does not have maximum volume growth in Lemma \ref{l-limit-comparison-1}
%? \end{enumerate}

\appendix
\section{}
In this appendix, we collect some basic facts for $U(n)$-invariant \K metrics on $\C^n$.

\begin{thm}\label{appendixthm} \begin{enumerate}
            \item [(a)]     \text {\sc ([Wu-Zheng]  \cite{WZ})} Every smooth $U(n)$ invariant \K metric $g$ is generated by a function $\xi: [0, \infty)\to \R$ with
$\xi(0)=0$ such that if

 \bee
 h_\xi(r) :=Ce^{{\int^r_{0}-\frac{\xi(s)}{s}ds}}; \hspace{10pt}
f_\xi(r) :=\frac{1}{r}\int^r_{0}h_\xi(s)ds
 \eee
  where $h_\xi(0)=C>0$ and $f_\xi(0)=h_\xi(0)$, where $r=|z|^2$, then
 \bee
 g_{i\bar{j}}=f_\xi(r)\delta_{ij}+f_\xi'(r)\overline{z_i}z_j.
\eee
where $g_{i\bar{j}}$ are the components of $g$ in the standard coordinates
$z=(z_1,\dots,z_n)$ on $\mathbb{C}^n$. Moreover $g$ is complete if and only if

\bee
\int_0^\infty \frac{\sqrt h_\xi(s)}{\sqrt s}ds=\infty.
\eee
              \item[(b)]\text {\sc ([Wu-Zheng]  \cite{WZ})} Let $h=h_\xi$, $f=f_\xi$.  At the point $z=(z_1,0,\dots,0)$, relative to the orthonormal frame $
e_1=\frac{1}{\sqrt{h}}\p_{z_1}, e_i=\frac{1}{\sqrt{f}}\p_{z_i}, i\ge 2$, with respect to $g$,  the curvature tensors $ A=R_{1\bar{1}1\bar{1}}=\frac{\xi'}{h}$,
  $B=R_{1\bar{1}i\bar{i}}=\frac{1}{(rf(r))^2}\int_0^r \xi'(s) \lf(\int_0^t
G(s)ds\ri) dt $, $C=R_{i\bar{i}i\bar{i}}=2R_{i\bar{i}j\bar{j}}=\frac{2}{(rf(r))^2}\int_0^r G(s)\xi(s)dt$,
 where $2\leq i\neq j\leq n$ and these are the only non-zero components of the
curvature tensor at $z$ except those obtained from $A, B$ or $C$ by the symmetric properties of $R$.
              \item[(c)] \text {\sc([Wu-Zheng]  \cite{WZ}, Yang \cite{Y})} $g$ has positive (nonnegative) bisectional curvature if and only if $\xi'>0$ ($\xi'\ge0$). In particular,   $g$ has nonnegative bisectional curvature and complete, then $\xi\le 1$.
                  \item[(d)] \text{\rm (\cite{ChauLiTam})} A complete $U(n)$-invariant \K metric generated by $\xi$ on $\C^n$ has bounded curvature if and only if $|\frac{\xi'}{h_\xi}|$ is uniformly bounded.
            \end{enumerate}
\end{thm}

 \begin{rem}  By the proof of \cite{WZ,Y}, it is easy to see that similar to (b) in the above theorem, $g$ has nonpositive bisectional curvature if and only if $\xi'\le 0$.
\end{rem}
 \begin{rem}\label{r-comparison} If $g_1$ and $g_2$ are two smooth $U(n)$ invariant \K metrics on $\C^n$ generated by $\xi_1, \xi_2$ respectively, and if the corresponding functions $h_{\xi_1}$, $h_{\xi_2}$ satisfy  $h_{\xi_1}\ge h_{\xi_2}$, then $g_1\ge g_2$. Conversely, if $g_1\ge g_2$, then $h_{\xi_1}\ge h_{\xi_2}$. This can be seen by comparing the metrics at the points $(a,0,\dots,0)$.

\end{rem}


\begin{thebibliography}{1000}
%\bibitem{CW}E.  Cabezas-Rivas; B., Wilking, {\sl How to produce a Ricci Flow via Cheeger-Gromoll exhaustion} %arXiv:1107.0606 (2011).

\bibitem{C}  Cao, H.D.,
  {\sl Deformation of K\"ahler metrics to K\"ahler
Einstein metrics on compact Kahler manifolds},
  Invent. Math. \textbf{81} (1985),  359--372.
\bibitem{Cao-1992} Cao, H.-D., {\sl On Harnack's inequalities for the K\"ahler-Ricci flow}, Invent. Math. \textbf{ 109}  (1992),  no. 2, 247--263.



\bibitem{ChauLiTam}    Chau, A.;  Li, K.-F.;  Tam, L.-F., {\sl Deforming complete Hermitian metrics with unbounded curvature}, arXiv:1402.6722 (2014).



\bibitem{ChauTam2007}  Chau, A.;   Tam, L.-F., {\sl A survey on the K\" ahler-Ricci flow and Yau's uniformization conjecture}, Surveys in Differential Geometry Vol. 12, J. Differential Geom. (2008).

\bibitem{CT} Chau, A.;   Tam, L.-F., {\sl On a modified parabolic complex Monge-\A equation with applications}, Math. Z. \textbf{269} (2011), no. 3-4, 777--800.

\bibitem{CW} Cabezas-Rivas, E.; Wilking, B., {\sl How to produce a Ricci Flow via Cheeger-Gromoll exhaustion}, to appear in J. Eur. Math. Soc., arXiv:1107.0606 (2011).


 \bibitem{CZ2} Chen, B.-L.; Zhu, X.-P., {\sl Uniqueness of the Ricci flow on complete noncompact manifolds} J. Differential Geom.\textbf{ 74} (2006), no. 1, 119--154.
\bibitem{CZ} Chen, B.-L.; Zhu, X.-P., {\sl Volume growth and curvature decay of positively curved K\"ahler manifolds.} Q. J. Pure Appl. Math. 1 (2005), no. 1, 68–108.
\bibitem{F} Fan, X.-Q., {\sl  A uniqueness result of \K Ricci flow with an application}, Proc. Amer. Math. Soc.  \textbf{135}  (2007),  no. 1, 289--298 (electronic).

\bibitem{GT} Giesen, G. ; Topping, P.M.,  {\sl Existence of Ricci flows of incomplete surfaces}, Comm. Partial Differential Equations \textbf{36} (2011), no. 10, 1860--1880.


 \bibitem{HuangTam}  Huang, S.; Tam, L.-F., {\sl \KR  flow with unbounded curvature}, preprint, arXiv:1506.00322.

\bibitem{KL} Koch, H.; Lamm, T.  Geometric flows with rough initial data. Asian J. Math.,
16(2): 209-235, 2012.

%\bibitem{NT} Ni, L.; Tam, L.-F., {\sl \KR flow and the Poincar\'e-Lelong equation}, Comm. Anal. Geom. \textbf{12} %(2004), no. 1-2, 111--141.


 \bibitem{NiTam2004} Ni, L.; Tam, L.-F., {\sl \KR flow and the Poincar\'e-Lelong equation}, Comm. Anal. Geom. \textbf{12} (2004), no. 1-2, 111--141.

 \bibitem{NiTam2013} Ni, L.; Tam, L.-F., { Poincar\'e-Lelong equation via the Hodge-Laplace heat equation,} Compos. Math.  \textbf{  149}  (2013),  no. 11, 1856--1870.


\bibitem{Shi1} Shi, W.-X., {\sl Deforming the metric on complete Riemannian manifolds}, J. Differential
Geom. 30 (1989), no. 1, 223--301.

    \bibitem{Shi2} Shi, W.-X., {\sl Ricci Flow and the uniformization on complete non compact \K manifolds}, J. of Differential Geometry \textbf{45} (1997), 94--220.

\bibitem{Si} Simon, M., {\sl Deformation of $C^0$ Riemannian metrics in the direction of their Ricci curvature},
Comm. Anal. Geom. \textbf{10}(2002), no. 5, 1033--1074.

\bibitem{SS1}Schn\"urer, O. C. ; Schulze, F. ; Simon, M., {\sl  Stability of Euclidean space under Ricci
flow}, Comm. Anal. Geom. \textbf{16} (2008), no. 1, 127--158.

\bibitem{SS2} Schn\"urer, O. C. ; Schulze, F. ; Simon, M.,{\sl Stability of hyperbolic space under Ricci flow},  Comm. Anal. Geom. \textbf{19} (2011), no. 5, 1023-1047.

\bibitem{T} Tam, L.-F.,  {\sl Exhaustion functions on complete manifolds}, Recent advances in geometric analysis,  Adv. Lect. Math. (ALM), \textbf{11}, 211--215, Int. Press, Somerville, MA, 2010.


\bibitem{To} Topping, P.M.,  {\sl Ricci flows with unbounded curvature}, Proceedings of the International Congress of Mathematicians, Seoul 2014, arXiv:1408.6866.

    \bibitem{WZ}Wu, H.-H ; Zheng, F., {\sl  Examples of positively curved complete \K manifold}, Geometry and Analysis Volume I, Advanced Lecture in Mathematics \textbf{17}, Higher Education Press and International Press, Beijing and Boston, 2010, pp. 517--542.

\bibitem{Y} Yang, B., {\sl On a problem of Yau regarding a higher dimensional generalization of the Cohn-Vossen inequality}, Math. Ann. \textbf{355}(2) (2013), 765-781.

\bibitem{YZ} Yang, B.;  Zheng, F., {\sl $U(n)$-invariant \KR flow with non-negative curvature}, Comm. Anal. Geom .\textbf{ 21}, (2013)
no. 2, 251--294.
 \end{thebibliography}
 \end{document}